\documentclass[10pt,oneside,a4paper]{article}

\usepackage{hyperref}
\usepackage{mathrsfs}
\usepackage{amsmath}
\usepackage{amsthm}
\usepackage{amssymb}
\usepackage{graphicx}
\usepackage{subfigure}

\newtheorem{theorem}{Theorem}
\newtheorem{lemma}[theorem]{Lemma}
\theoremstyle{remark}
\newtheorem{remark}[theorem]{Remark}

\numberwithin{equation}{section}

\newcommand{\R}{\mathbb{R}}

\newcommand{\Z}{\mathbb{Z}}
\newcommand{\C}{\mathbb{C}}

\newcommand{\ZZ}{{\mathbb{Z}}}
\newcommand{\ri}{{\mathrm{i}}}
\newcommand{\re}{{\mathrm{e}}}

\newcommand{\AckNIDispHyd}[1]{#1 would like to thank the Isaac Newton Institute for Mathematical Sciences for support and hospitality during programme \emph{Dispersive hydrodynamics: mathematics, simulation and experiments, with applications in nonlinear waves}, when work on this paper was undertaken. This work was supported by EPSRC Grant Number EP/R014604/1.}

\title{The role of periodicity in the solution of third order boundary value problems}
\author{
    B. Pelloni* \& D. A. Smith\textsuperscript{\textdagger} \\
    \footnotesize
    *Heriot-Watt University \& Maxwell Institute for the Mathematical Sciences, \\
    \footnotesize
    Edinburgh, Scotland
    \href{mailto:b.pelloni@hw.ac.uk}{b.pelloni@hw.ac.uk} \\
    \footnotesize
    \textsuperscript{\textdagger}Yale-NUS College \& National University of Singapore, Singapore \\
    \footnotesize
    \href{mailto:dave.smith@yale-nus.edu.sg}{dave.smith@yale-nus.edu.sg}
}

\date{\today}

\begin{document}
\maketitle

\begin{abstract}
    In this short paper, we elucidate how the solution of certain illustrative boundary value problems for the Airy equation $u_t+u_{xxx}=0$ on $[0,1]$ can be expressed as a perturbation of the solution of the purely periodic problem. The motivation is to understand the role boundary conditions play in the properties of the solution. This is particularly important in related work on the solution of linear dispersive problems with discontinuous initial data and the phenomena of revivals and fractalization.
\end{abstract}

\subsubsection*{Keywords}
Initial boundary value problems,
Regularity of solutions,
Fokas transform method,
Revivals and fractalization.

\section{Introduction}\label{introduction}
Linear dispersive equations such as the free-space Schr\"odinger or Airy equations, respectively
\[
iu_t-u_{xx}=0\qquad \mbox{or}\qquad u_t+u_{xxx}=0,\qquad \mbox{where } u=u(x,t),
\]
are  important models in the mathematical modelling of reality, as they constitute a powerful way to capture the dominant linear behaviour driving the time evolution of many physical phenomena that propagate in a wave-like manner.
We consider them in one space dimension, so $x\in\R$, with $t>0$ denoting time.  These are the simplest equations that, mathematically, capture the main features of the behaviour of even- and odd-order linear dispersive problems.

Our physical reality demands that we pose these equations on a finite interval. We will assume in all that follows that $x\in [0,1]$, and that we are given initial and boundary conditions that yield a well-posed problem that admits a unique solution.

\medskip
While the method of solution  of these linear boundary value problems, via separation of variables or eigenfunction expansion,  has been a standard tool of the mathematical trade for a very long time, these techniques are not universally applicable. In particular,  they rely on the full and explicit knowledge of the eigenstructure of the spatial linear differential operator. However, when the boundary conditions are such that this operator is not self-adjoint, this eigenstructure may not be known or easily determined. This is particularly true for odd-order operators, which is the case we focus on in this note.

We also remark that care must be taken when the given initial or boundary data are not sufficiently regular, as it is then necessary  to interpret the solution and its representation in a suitably weak sense. We will not dwell on this aspect in this note, but it is an important consideration for some of the applications, notably the study of revivals and weak revivals,  that motivated the considerations in this paper.

More precisely, Our motivation is recent work on the so-called {\em Talbot effect}, or {\em revival} phenomenon~\cite{BFP,olver2018revivals}.
This phenomenon, first described experimentally in the mid 1850's by scientist and pioneer of photography, William Henry Fox Talbot~\cite{Tal1836a} and rediscovered several times since in other dispersive systems (notably by Olver~\cite{olver2010dispersive}), occurs when an initial datum with jump discontinuities is propagated periodically. What happens then is that the behaviour of the solution at times that are rational multiples of a certain quantity related to the length of the interval, that we call rational times,  is markedly different from the behaviour at generic times.  At rational times, the solution is a superposition of translated and dilated copies of the initial conditions, so in particular it is spatially discontinuous, while at generic times the solution is spatially continuous (albeit nowhere differentiable).
See~\cite{smith2020revival} for a recent survey of these phenomena.

While studying whether this phenomenon persists in more general situations, for example for more general boundary conditions, or for nonlinear dispersive PDEs, it is natural to consider what echo is left by  the solution of the linear periodic problem, for which the phenomenon of revivals occurs. This is  the question that we consider, for selected examples, in this paper.

\section{A motivating example and the Unified Transform of Fokas}
One particularly natural and surprisingly hard problem, illustrating the limitations of the eigenfunction expansion technique, is the following.

Assume $u(x,t)$ solves the following {\em Dirichlet type} boundary value problem on $[0,1]$ for the Airy equation:
\begin{align}
\label{pseD}
&u_t+u_{xxx}=0,\qquad &&x\in(0,1), \;t>0, \nonumber
\\
&u(x,0)=f(x),\qquad &&
x\in(0,1),\;
\\
&u(0,t)=u(1,t)=\partial_xu(1,t)=0,\qquad&& t>0.\nonumber
\end{align}
Here, and in what follows, we do not make specific assumptions on the regularity of the initial datum $f(x)$. If $f(x)$ is sufficiently smooth, we might expect the solution representation to be valid pointwise, though this depends on the compatibility at the corners $(0,0)$ and $(1,0)$  (see also e.g. \cite{trogdon}).  As already remarked, the regularity assumptions on $f(x)$  can be relaxed by giving a weaker definition of solution.

It is remarkable that, while the associated differential operator has an infinite number of discrete eigenvalues, the associated eigenfunctions do not form an unconditional basis, so that there is no generalised Fourier series representation for the solution of this problem; separation of variables cannot yield a solution of this problem. Although this was formally established as far back as 1915~\cite{jackson}, the knowledge fell into obscurity until the problem was rediscovered and solved  within a more general setting~\cite{pelloni2005spectral,fokas2005transform,smithfokas}.

However, the problem can be fully and effectively solved  using the Fokas transform method \cite{fokas1997}. Indeed, using this approach it can be shown that  this problem has a unique solution $u(x,t)$, \cite{pelloni2005spectral}.
The solution admits the explicit contour integral representation
\begin{multline*}
2\pi u(x,t)=\int_{\mathbb R} \re^{ikx+ik^3t}\widehat{u_0}(k)dk \\
+\int_{\Gamma^+}\re^{ikx+ik^3t}\frac {\zeta^+(k)}{\Delta(k)}dk
+\int_{\Gamma^-}\re^{ik(x-1)+ik^3t}\frac {\zeta^-(k)}{\Delta(k)}dk,
\end{multline*}
where $\Gamma^{\pm}$ are the contours in $\C^\pm$ defined as the locus of $\Re(i k^3)=0$, $\widehat{u_0}(k)$ denotes the Fourier transform of $u_0(x)$, and $\zeta^\pm(k)$, $\Delta(k)$ are entire functions of $k$ only, fully determined by the function $u_0(x)$.  The zeros of $\Delta$  are the cube roots of the eigenvalues of the spatial differential operator, and it can be shown that they are not on the integration contours~\cite[\S A]{pelloni2005spectral}. Hence the integrals are well defined.

\medskip
This leaves open the question of how this problem relates to the simpler case of periodic boundary conditions, whose solution has a classical representation in terms of a Fourier series.  This is the question we consider below, for this example as well as for examples of boundary conditions that couple the two ends of the interval $[0,1]$.

\section{Non-periodic boundary value problems}

Let $u(x,t)$ denote the solution of a given boundary value problem of the form
\begin{subequations} \label{eqn:genbvp}
\begin{align}
&u_t+u_{xxx}=0,\qquad &&x\in(0,1), \;t>0,
\\
&u(x,0)=f(x),\qquad &&
x\in(0,1),\;
\\ &\mbox{\em three homogeneous boundary conditions on } u(x,t). \label{genbvp.BC}
\end{align}
\end{subequations}
We assume that the prescribed boundary conditions are not periodic ones, and that they are such that the solution exists and is unique. For the Airy equation, the boundary conditions for which this well-posedness holds are characterised in \cite{smith2012}. In particular,  this is the case for the illustrative examples we examine below.  We stress that we rely crucially on the Fokas transform approach to guarantee that such existence results hold for  the examples given.

\medskip
To relate the solution of such a boundary value problem with the solution of the periodic problem, we consider a natural decomposition of the solution $u(x,t)$.
Namely, let $v(x,t)$ denote the solution of the purely periodic problem, with the same initial condition $v(x,0)=f(x)$,  so that $v(x,t)$ satisfies
\begin{align}\label{vprob}
&v_t+v_{xxx}=0,\;\; x\in(0,1), \;t>0;
\qquad
v(x,0)=f(x),\;\;x\in(0,1);
\\
&\partial_x^jv(0,t)=\partial_x^j v(1,t),\quad j=0,1,2,\quad t>0.
\end{align}

The function $v(x,t)$ admits the following explicit representation as a Fourier series, pointwise if $v(x,t)$ is sufficiently smooth (e.g. at least H\"older continuous), and in  $L^2[0,1]$ otherwise:
\begin{equation} \label{eqn:vsoln}
v(x,t)=\sum_{n\in\ZZ} \re^{ik_nx+ik_n^3t}\hat{f}(k_n), \qquad k_n=2\pi n.
\end{equation}
The spectral structure of this problem is entirely understood: the spectrum is fully discrete and given by $\{k_n, \;n\in\Z\}$, while the associated eigenfunctions $\{\re^{ik_n x},\;n\in\mathbb Z\}$ form a complete basis with respect to the $L^2$ Hilbert structure.

\medskip
We then define the auxiliary function $w$ as
\begin{equation}\label{wdef}
w(x,t):=u(x,t)-v(x,t).
\end{equation}
The function $w$ is fully determined by the given functions $u$ and $v$, hence all its boundary values are known.   This function encodes information on how the given boundary conditions change the nature of the solution when compared with the solution of the periodic problem.

In each of the following sections, we will select different boundary conditions~\eqref{genbvp.BC}, and examine the properties of the resulting $w$, in particular its regularity, with the aim of characterising $u$ as a $w$-regularity perturbation of $v$.

For most of the $u$ problems considered here, it is appropriate to select $v$ as the solution of a periodic problem and  then $w$ can be viewed as the solution of a problem with zero initial condition and  with boundary conditions that are, formally, inhomogeneous periodic or quasiperiodic conditions. To emphasize the role of the inhomogeneities in the boundary conditions of $w$, we describe such problems as ``forced (quasi)periodic''.

In the last case, we choose $v$ as the solution of a quasiperiodic, rather than periodic, problem.

In table~\ref{tbl:Summary} we summarise the type of boundary value problem given for $u$ and selected for $v$. Once these are given, the function $w$ is fixed but the boundary value problem for it can be given either in terms of the boundary values of $u$ or of $v$, resulting in a different problem for $w$; the third column in this table  summarises what particular problem is selected for $w$ for each of the examples we treat in this paper.

\begin{table}
    \centering
    \setlength{\tabcolsep}{7pt}
    \renewcommand{\arraystretch}{1.2}
    \begin{tabular}{lllll}
      \textbf{$u$ problem} &  \textbf{$v$ problem} & \textbf{$w$ problem} &  \textbf{result} & \textbf{eg} \\ \hline
       Dirichlet type & periodic & forced periodic &  thm~\ref{thm:main} & \S\ref{ssec:Uncoupled} \\ \cline{5-5}
         &  &  &  & \S\ref{ssec:Mixed} \\ \cline{1-5}
         $u(0,t)=u(1,t)$ & periodic & forced periodic  &  thm~\ref{thm:main} & \S\ref{ssec:Pseudo} \\ \cline{1-5}
       quasiperiodic &periodic  & forced quasiperiodic &  rmk~\ref{remQuasiRevivals} & \S\ref{ssec:Quasi} \\ \hline
  $u(0,t)=\re^{i\theta}u(1,t)$ &      quasiperiodic & forced quasiperiodic &  rmk~\ref{remMakevQuasi} & \\ \hline
    \end{tabular}
    \caption{Summary of problems, examples, and results.}
    \label{tbl:Summary}
\end{table}

\subsection{Uncoupled BC of Dirichlet type} \label{ssec:Uncoupled}

Here we consider the problem \eqref{pseD} for $u(x,t)$. In this case,  the function $w(x,t)$ satisfies
\begin{subequations}\label{wsys}
\begin{align}
\label{wsys.PDE.IC}
&w_t+w_{xxx}=0,\;\; x\in(0,1), \;t>0; &w(x,0)=0,\;\;&x\in(0,1);
\\
&w(0,t)=w(1,t),  \quad& t>0;&\\
& \partial_xw(0,t)=\partial_xw(1,t)+h_1(t),  \quad& t>0;&\\
& \partial_{xx}w(0,t)=\partial_{xx}w(1,t)+h_2(t), \quad &t>0.&
\end{align}
\end{subequations}
The known, smooth functions $h_1(t)$ and $h_2(t)$ are given by
\begin{align*}
&h_1(t)=\partial_{x}u(0,t),\\
&h_2(t)=\partial_{xx}u(0,t)-\partial_{xx}u(1,t).
\end{align*}
Hence $w$ can be regarded as the solution of a forced periodic problem, with zero initial condition.

\begin{lemma}\label{basiclem}
The function $w(x,t)$ that solves~\eqref{wsys} is a continuous function of $x$ and admits the representation
\begin{equation}\label{pseudoDir}
w(x,t)=\sum_{n\in \Z} \re^{ik_nx+ik_n^3t}\left[ik_nH_1(k_n,t)+H_2(k_n,t)\right],
\end{equation}
with $k_n = 2\pi n$,
\begin{equation}\label{bigH}
H_1(k,t)=\int_0^t\re^{-ik^3 s}h_1(s)ds,\qquad H_2(k,t)=\int_0^t\re^{-ik^3 s}h_2(s)ds.
\end{equation}
\end{lemma}

\begin{proof}
Consider the Fourier transform of $w(x,t)$ on $[0,1]$, defined by
$$
\hat w(k,t)=\int_0^1\re^{-ikx}w(x,t)dx.
$$
Then, using the Fourier transform and integration by parts, the PDE for $w(x,t)$ yields the following ODE for $\hat w(k,t)$:
\begin{multline*}
\bigl(\re^{-ik^3 t}\hat w(k,t)\bigr)_t
=-[k^2w(0,t)-ik w_x(0,t)-w_{xx}(0,t)] \\
+\re^{-ik}[k^2w(1,t)-ikw_x(1,t)-w_{xx}(1,t)].
\end{multline*}
We set
\begin{subequations} \label{eqn:defnFG}
\begin{align}
F(k,t)&=\int_0^t\re^{-ik^3s}[k^2w(0,s)-ikw_x(0,s)-w_{xx}(0,s)]ds, \\
G(k,t)&=\int_0^t\re^{-ik^3s}[k^2w(1,s)-ikw_x(1,s)-w_{xx}(1,s)]ds.
\end{align}
\end{subequations}
Then, since $\hat w(k,0)=0$, the solution of the ODE is given by
$$
\hat w(k,t)=
\re^{ik^3t}\bigl[-F(k,t)+\re^{-ik}G(k,t)\bigr].
$$
Using the boundary conditions, we find
$$
F(k,t)=G(k,t)-ikH_1(k,t)-H_2(k,t),
$$
with $H_1(k,t)$, $H_2(k,t)$ defined in \eqref{bigH}.
Hence
$$
\hat w(k,t)=
\re^{ik^3t}\bigl[(\re^{-ik}-1)G(k,t)+ikH_1(k,t)+H_2(k,t)\bigr].
$$
Evaluating this expression at $k=k_n$ for $n\in\Z$, we obtain
$$
\hat w(k_n,t)=\re^{ik_n^3t}\bigl[ikH_1(k_n,t)+H_2(k_n,t)\bigr].
$$
Hence, using the Fourier series representation
\begin{equation}\label{frep}
w(x,t)=\sum_{n\in\Z}\re^{ik_nx}\hat w(k_n,t),
\end{equation}
we arrive at the representation \eqref{pseudoDir} for $w(x,t)$.

Since the functions $h_1(t)$ and $h_2(t)$ are differentiable, the coefficients in the series \eqref{pseudoDir} decay at least as $1/k_n^2$, which guarantees the continuity of $w(x,t)$ with respect to $x$.
\end{proof}

\begin{remark}\label{rem1}
    It is crucial in the argument above that the first boundary term in the Fourier transform of the PDE, namely the two terms $k^2 w(0,t)$ and $k^2 w(1,t)$, vanish. Indeed the presence of either of these terms would make the decay in $k$ of the coefficients in the Fourier series~\eqref{frep} too slow to guarantee that the solution is continuous.
\end{remark}

From all this we infer that the solution $u(x,t)$ of the original Dirichlet type problem has the formal representation
\begin{multline*}
u(x,t)=v(x,t)+w(x,t)=\sum_{n\in\Z}\re^{ik_nx+ik_n^3t}\hat f(k_n) \\
+\sum_{n\in\Z}\re^{ik_nx+ik_n^3t}
\int_0^t\re^{-ik_n^3s}(ik_n u_x(0,s)+u_{xx}(0,s)-u_{xx}(1,s))ds.
\end{multline*}

If we did not have the a priori knowledge that the  function $u(x,t)$ exists and is unique, the above would be purely a formal expression, with nothing new to offer; it would not yield a way to represent $u(x,t)$ effectively.
However, because we do have wellposedness of~\eqref{pseD}, expressing $u(x,t)$ in this way gives information on how its regularity properties depend on the regularity of initial and boundary conditions.

For the solution $v(x,t)$ of the purely periodic problem,  the regularity depends only on the functional class of $f(x)$. It is less known that if $f(x)$ is only of bounded variation, but not continuous, the regularity of the solution remains in the same class at certain values of the time in a dense set of measure $0$, but improves for almost all $t$. This is known in the context of the periodic problem as the phenomenon of {\em revivals}.

The second sum on the right hand side of the expression above for $u(x,t)$ conveys information on how the regularity is affected by the (homogeneous) boundary conditions.
Lemma~\ref{basiclem} implies that the second term is always continuous as a function of $x$.
Therefore, $u$ itself is a continuous perturbation of $v$.

\subsection{Mixed BC of Dirichlet type} \label{ssec:Mixed}

If the boundary conditions for $u(x,t)$ include the Dirichlet-type condition   $u(0,t)=u(1,t)=0$,  plus another condition possibly coupling the ends of the interval $[0,1]$, the analysis of the previous examples remain essentially unaltered. For example, if the third condition is $u_x(0,t)=\gamma u_x(1,t)$, for some $\gamma\in(0,1)$, the  argument detailed above follows through with
$$
H_1(k,t)=\int_0^t (\gamma-1)u_x(1,s)\re^{-ik^3s}ds,
$$
and the value of $H_2(k,t)$ given by the latter of equations~\eqref{bigH}.
Therefore, the same conclusion can be drawn: regardless of the regularity (or lack of regularity) of $u(x,t)$ as a function of $x$, the function $w(x,t)$ is continuous in $x$, so $u$ is a continuous perturbation of $v$.

\medskip
Note that, unlike the previous example, in this case the spatial operator admits an $L^2$ basis of eigenfunctions, even though the eigenvalues cannot be determined explicitly other than as roots of a transcendental equation. Using the Fokas transform approach, the associated generalised Fourier series can be determined by a contour deformation technique~\cite{pelloni2005spectral,smithfokas}.

\subsection{Coupled BC: pseudo-periodic} \label{ssec:Pseudo}
We now turn to boundary conditions that couple the endpoints of the interval $[0,1]$, and assume that the given boundary conditions for $u(x,t)$ are the pseudo-periodic  conditions
\begin{equation}\label{bcpp}
\beta_j\partial_x^j u(0,t)= \partial_x^j  u(1,t),\quad j=0,1,2,\qquad \beta_j\in\C.
\end{equation}
Conditions need to be imposed on the $\beta_j$'s to ensure the problem is well-posed, see \cite{smith2012}. We assume this to be the case.

The function $w(x,t)$ satisfies, along with a zero initial condition, the  boundary conditions
\begin{equation} \label{eqn:PP.wperiodic.BC}
\partial_x^jw(0,t)=\partial_x^jw(1,t)+h_j(t),  \quad j=0,1,2,
\end{equation}
where, in this case,
$$
h_j(t)=(1-\beta_j)\partial_x^ju(0,t), \qquad j=0,1,2.
$$
A lemma entirely analogous to Lemma~\ref{basiclem} yields for $w(x,t)$ the representation
\begin{equation}\label{pseudoper}
w(x,t)=\sum_{n\in \Z} \re^{ik_nx+ik_n^3t}\left[-k_n^2H_0(k_n,t)+ik_nH_1(k_n,t)+H_2(k_n,t)\right],
\end{equation}
where $k_n=2n\pi$ and
$$
H_j(k,t)=\int_0^t\re^{-ik^3s}(1-\beta_j )\partial_x^ju(0,s)ds.
$$

As noted in Remark~\ref{rem1}, the function $w(x,t)$ now has coefficients that can be guaranteed to decay only as $n^{-1}$, and therefore it may have lower regularity than the given initial datum $f(x)$, assumed H\"older continuous.
However, if it so happens that $\beta_0=1$, then $H_0=0$, so the coefficients in equation~\eqref{pseudoper} decay like $n^{-2}$, $w$ is continuous and, as before, $u$ is a continuous perturbation of $v$.

\medskip
The results presented in the previous sections can be summarised and generalised as the following theorem.
\begin{theorem} \label{thm:main}
    Suppose problem~\eqref{eqn:genbvp} is wellposed and $u$ is its solution.
    If the given linearly independent boundary conditions~\eqref{genbvp.BC} are either
    \[
        u(0,t) = 0, \qquad u(1,t) = 0, \qquad \mbox{one other boundary condition on }u
    \]
    or
    \[
        u(0,t) = u(1,t), \qquad \mbox{two other boundary conditions on }u,
    \]
    then
    $u(x,t)=v(x,t)+w(x,t)$,
    for $v$ the solution~\eqref{eqn:vsoln} of periodic problem~\eqref{vprob} and $w$ a continuous function of $x$.
\end{theorem}

\subsection{An outlier: quasi-periodic BC} \label{ssec:Quasi}

For the particular case that $\beta_0=\beta_1=\beta_2$ in example~\eqref{bcpp}, known as the quasi-periodic case, one can pursue an alternative argument to give some interesting qualitative information about the function $u(x,t)$.
This information is consistent with the fact that quasi-periodic problems for the Airy equation do not in general exhibit the phenomenon of weak revivals~\cite{BFP}.
It is also consistent with the fact that the spectral structure of the quasi-periodic spatial operator can be easily derived by a shift on the structure of the periodic operator, and is well known.
 
\medskip
Assume that the given boundary conditions for $u(x,t)$ are the quasi-periodic conditions
\begin{equation}\label{bcqp}
\partial_x^j u(0,t)= \re^{i\theta}\partial_x^j  u(1,t),\quad j=0,1,2, \quad \theta\in\R.
\end{equation}
The function $w(x,t)$ satisfies, along with zero initial conditions, the boundary conditions
\begin{equation} \label{eqn:QP.wquasiperiodic.BC}
\partial_x^j w(0,t)= \re^{i\theta}\partial_x^j w(1,t)+(1-\re^{-i\theta}) \partial_x^j v(1,t),   \;\;t>0,
\end{equation}
where $v(x,t)$ is the solution of the purely periodic problem \eqref{vprob}.
It is still true that $w$ obeys boundary conditions~\eqref{eqn:PP.wperiodic.BC} with $\beta_0=\beta_1=\beta_2=\re^{-\ri\theta}$, but we shall make use of the alternative characterisation~\eqref{eqn:QP.wquasiperiodic.BC} in the following argument.

\begin{lemma}
The function $w(x,t)$ solution of~\eqref{wsys.PDE.IC},~\eqref{eqn:QP.wquasiperiodic.BC} admits the representation
\begin{equation}\label{qprep}
w(x,t)=\sum_{n\in \Z} \re^{i(k_n-\theta)x+i(k_n-\theta)^3t}\int_0^t(1-\re^{i\theta})V(k_n-\theta,s)\re^{-i(k_n-\theta)^3s}ds,
\end{equation}
with $k_n=2n\pi$ and
\begin{equation}\label{bigv}
V(k,t)=k^2v(0,t)-ikv_x(0,t)-v_{xx}(0,t).
\end{equation}
\end{lemma}

\begin{proof}
Consider the Fourier transform of $w(x,t)$, which satisfies as before the following ODE:
$$
\bigl(\re^{-ik^3 t}\hat w(k,t)\bigr)_t=-F(k,t)+\re^{-ik}G(k,t),
$$
with $F(k,t)$ and $G(k,t)$ given by equations~\eqref{eqn:defnFG}.
Then, since $\hat w(k,0)=0$, the solution of the ODE is given by
$$
\hat w(k,t)=
\re^{ik^3t}\bigl[-F(k,t)+\re^{-ik}G(k,t)\bigr].
$$
Using the boundary conditions, we find
$$
\re^{-i\theta}F(k,t)=G(k,t)+(1-\re^{-i\theta})\int_0^t\re^{-ik^3s}V(k,s)ds,
$$
with $V(t,k)$ defined in \eqref{bigv}.
Hence
$$
\hat w(k,t)=
\re^{ik^3t}\left[(\re^{-i(k+\theta)}-1)F(k,t)+(\re^{-i(k+\theta)}-\re^{-ik})\int_0^t\re^{-ik^3s}V(k,s)ds\right].
$$
Evaluating this expression at $k=k_n-\theta$ for $n\in\Z$, we obtain
$$
\hat w(k_n-\theta,t)=\re^{i(k_n-\theta)^3t}(1-\re^{i\theta})\int_0^t\re^{-i(k_n-\theta)^3s}V(k_n-\theta,s)ds.
$$
We now invert this to obtain the generalised Fourier series expression
\begin{equation}\label{frep2}
w(x,t)=\sum_{n\in\Z}\re^{i(k_n-\theta)x}\hat w(k_n,t),
\end{equation}
which is the representation~\eqref{qprep} for $w(x,t)$.
\end{proof}

From all this, we infer that the solution $u(x,t)$ of the original quasi-periodic problem has the formal representation
\begin{multline} \label{eqn:urep3}
u(x,t)=v(x,t)+w(x,t)
=\sum_{n\in\Z}\re^{ik_nx+ik_n^3t}\hat f(k_n) \\
+\sum_{n\in\Z}\re^{i(k_n-\theta)x}\re^{i(k_n-\theta)^3t}\int_0^t(1-\re^{i\theta})\re^{-i(k_n-\theta)^3s}V(k_n-\theta,s)ds.
\end{multline}
The function $V(k,t)$ is made up of the boundary values of the $x$-periodic function $v(x,t)$.
Knowledge of this function as well as the characterisation of the eigenvalues of the spatial operator, is enough to represent $u(x,t)$ effectively, and $V(k,t)$ can easily be calculated from representation~\eqref{eqn:vsoln}.

\begin{remark}\label{remQuasiRevivals}
    Note that the presence of the term $k_n^2 v(0,t)$ in the definition of the generalised  Fourier coefficient $\widehat{w}(k,t)$ implies that the convergence of the series for $w(x,t)$ is slow and not uniform; unless $v(0,t)=0$, this term implies a no better regularity of $w$ than that of the solution of the purely periodic problem.

    Note also that the second term on the right of representation~\eqref{eqn:urep3} contains the exponential $\re^{i(k_n-\theta)x+i(k_n-\theta)^3t}$ which is  both space- and time-periodic with a period congruent to $\theta+\mathbb Q$, while $v(x,t)$ is periodic with period in $\mathbb Q$. Therefore if $\theta \notin\mathbb Q$, the function $u(x,t)$  cannot have any periodicity property.
    This confirms the result of~\cite{BFP}, namely the fact that this quasi-periodic problem, surprisingly, does not exhibit revivals if $\theta \notin\mathbb Q$.
\end{remark}

\begin{remark}\label{remMakevQuasi}
    Suppose the boundary conditions for $u$ are
    \begin{equation}\label{bcuquasizero}
        u(0,t)=\re^{i\theta}u(1,t), \qquad \mbox{two other boundary conditions on }u,
    \end{equation}
    with $\theta\in\mathbb R$ but, to avoid the regime already covered by theorem~\ref{thm:main}, suppose $\theta$ is not an even integer multiple of $\pi$.
    Suppose this problem for $u$ is wellposed.
    Note that this includes certain pseudoperiodic problems~\eqref{bcpp}, but not all wellposed such problems.
    
    We can make the decomposition $u(x,t)=v(x,t)+w(x,t)$ with $v$ satisfying the quasiperiodic problem
    \begin{align*}
        &v_t+v_{xxx}=0,\;\; x\in(0,1), \;t>0;
        \qquad
        v(x,0)=f(x),\;\;x\in(0,1);
        \\
        &\partial_x^j v(0,t)= \re^{i\theta}\partial_x^j v(1,t),\quad j=0,1,2,\quad t>0,
    \end{align*}
    and $w$ satisfying the boundary forced quasiperiodic problem
    \begin{align*}
        &w_t+w_{xxx}=0,\;\; x\in(0,1), \;t>0;
        \qquad
        w(x,0)=0,\;\;x\in(0,1);
        \\
        &\partial_x^j w(0,t)= \re^{i\theta}\partial_x^j w(1,t) + h_j(t),\quad j=0,1,2,\quad t>0,
    \end{align*}
    in which
    \[
        h_0(t) = 0, \qquad h_1(t) = u_x(0,t)-\re^{\ri\theta}u_x(1,t), \qquad h_2(t) = u_{xx}(0,t)-\re^{\ri\theta}u_{xx}(1,t).
    \]
    Then, as discussed above, the (non)existence of revivals for $v$ is determined by the (ir)rationality of $\theta$ and, using an argument exactly paralleling the proof of lemma~\ref{basiclem}, $w$ is continuous.
    Therefore, $u$, being a continuous perturbation of $v$, exhibits continuous perturbations of revivals if and only if $\theta\in\mathbb Q$.
    This is the analogue of theorem~\ref{thm:main} for boundary conditions~\eqref{bcuquasizero}.
\end{remark}

\subsubsection*{Conclusion}

We have embedded the solution of the periodic problem in the solution of certain classes of homogeneous boundary value problems to determine how the boundary conditions perturb the qualitative properties of the periodic solution.

For homogeneous Dirichlet type separated boundary conditions, and for boundary conditions that describe continuous extension from $[0,1]$ to $(-\infty,\infty)$, we found that the remaining one or two boundary conditions add a component that superimposes a continuous function of $x$ onto the periodic solution, irrespective of the overall $x$ regularity of the full solution.

On the other hand, in the case of some particular quasi-periodic problems, which in the case of second order problems can always be recast in terms of periodic boundary conditions, this approach confirms that while the solution depends only on the boundary values of the periodic solutions, the boundary conditions not only  add a less regular component to the purely periodic solution, but also that
the interaction between the periodic and the non-periodic part of the solution can silence completely the echo of periodicity.

These remarks have particularly significant consequences in case of low-regularity initial data and the phenomenon of weak revivals. This is explored further in~\cite{BFPS}.

\section*{Acknowledgement}

\AckNIDispHyd{The authors}

\bibliographystyle{amsplain}
\bibliography{./references}

\end{document}